\documentclass[12pt]{amsart}
\usepackage{amsfonts}
\usepackage{amsmath}
\usepackage{amsxtra}
\usepackage{amssymb,latexsym}
\usepackage[mathcal]{eucal}
\usepackage{amscd}
\usepackage[all]{xy}

\makeindex

\newtheorem{theo}{{\bfseries Theorem}}[section]
\newtheorem{prop}[theo]{{\bfseries Proposition}}
\newtheorem{lem}[theo]{{\bfseries Lemma}}
\newtheorem{cor}[theo]{{\bfseries Corollary}}
\newtheorem{df}[theo]{{\bfseries Definition}}

\def \N {\mathbb N}

\def \Z {\mathbb Z}
\def \R {\mathbb R}
\def \Q {\mathbb Q}

\def \I {\mathcal I}

\def \J {\mathcal J}

\def \T {\mathcal T}

\def \aa {{\mathbf a}}
\def \a {\alpha }
\def \b {\beta}

\def \t {\tau}

\usepackage{amssymb,latexsym}
\usepackage[mathcal]{eucal}
\usepackage{amscd}
\usepackage{amsmath}
\usepackage{graphicx}
\usepackage{amsfonts}
\usepackage{amscd}

\numberwithin{equation}{section}

\begin{document}

\title[Asymptotic Pairs]{\bfseries  Asymptotic Pairs for Interval Exchange Transformations}
\vspace{1cm}
\author{Ethan Akin, Alfonso Artigue and Luis Ferrari}
\address{Mathematics Department \\
    The City College \\ 137 Street and Convent Avenue \\
       New York City, NY 10031, USA     }
\email{ethanakin@earthlink.net, julia.saccamano@macaulay.cuny.edu}

\address{Departamento de Matemática y Estadística del Litoral, Cenur Litoral Norte, Universidad de la República, Gral. Rivera 1350, Salto, Uruguay}
\email{artigue@unorte.edu.uy}

\address{Departamento de Matemática y Aplicaciones, Cure, Universidad de la República, Maldonado, Uruguay}
\email{luisrferrari@gmail.com}

\date{October, 2020 }

\begin{abstract} We provide a simple description of asymptotic pairs in the subshift associated with an interval exchange transformation and show that,
under reasonably general conditions, doubly asymptotic pairs do not occur.
\end{abstract}

\keywords{}

\thanks{{\em 2010 Mathematical Subject Classification} }
\vspace{1cm}

\vspace{.5cm} \maketitle

\tableofcontents

\section{Introduction }\label{sec1}

Interval exchange transformations were introduced by Keane \cite{Ke}, see also \cite{V} for a survey.
Keane begins with a partition of $[0,1)$ by left-closed intervals  $\{ J_1,\dots, J_n \}$ and defines a bijection $T$ on $[0,1)$ which is a
translation on each interval. The itinerary map $\I : [0,1) \to \Omega = \{1, \dots, n \}^{\Z}$ is defined by $\I(x)_k = i$ when
$T^kx \in J_i$ for $k \in \Z$ (We will use, as usual, $\Z$ for the set of integers and $\N$ for the set of positive integers).
The cocountable set $R \subset (0,1)$ consisting of the complement of the set of orbits of the endpoints is called the set of regular points.
At points of $R$ the iterates of $T$ and the itinerary map $\I$ are continuous. The function $\I$ maps the homeomorphism $T$ on $R$
to the shift map $S$ on $\Omega$ where $S(\a)_i = \a_{i+1}$. Taking the closure of $\I(R)$ in $\Omega$ we obtain
a closed shift invariant subset $X$.


Under conditions which have become known as the Keane Condition and irreducibility,
Keane showed that the subshift $(X,S)$ is minimal, i.e. all orbits are dense.

It will be convenient to consider at the same time the dual interval exchange map $\tilde T$. This is defined using the right-closed
partition of $(0,1]$ $\{ \tilde J_1,\dots, \tilde J_n \}$ which have the same end-points as before and which uses the same translation maps.
The itinerary map $\tilde \I : (0,1] \to \Omega$  is defined as before. On $R$ we have $\I = \tilde \I$.

By using both we characterize the space $X$ as the union $ \I([0,1)) \cup \tilde \I((0,1])$.  In addition, we are able to provide a complete description of
the asymptotic pairs, under the assumption of irreducibility and the Keane Condition.

If $F$ is a homeomorphism on a compact metric space $K$ then  points $x,y \in K$ are \emph{positively asymptotic}
when $\lim_{i \to \infty} \ d(F^ix, F^iy) = 0$ and \emph{negatively asymptotic} when they are positively asymptotic for $F^{-1}$.
A pair is called \emph{doubly asymptotic} when it is both positively and negatively asymptotic.

Two points
$\a, \b \in \Omega$ are positively asymptotic (or negatively asymptotic) if and only if there exists $N \in \N$ such that $\a_i = \b_i$
for all $i \geq N$ (resp. for all $i \leq -N$).

From the characterization we obtain our main result.

\begin{theo}\label{theo1.1} If an interval exchange transformation is  irreducible, fully split and satisfies the Keane Condition, then no pair of
distinct points is doubly asymptotic for the associated subshift. \end{theo}

Using a different procedure, in \cite{Ki} King gives an example of a subshift without doubly asymptotic pairs.

\section{$T$ Intervals }\label{sec2}

By an \emph{interval} $I$ we mean a bounded, nonempty subset of $\R$ which is connected, or, equivalently,  $x_1 < x_2 \in I$ implies $y \subset I$
for all $y$ with $x_1 < y < x_2$  (no holes).
We call $I$ a \emph{positive interval} when its length $\ell(I)$ is positive. An interval which is not positive is just a singleton in $\R$.
There are four types of intervals. With $a < b$ the intervals $[a,b),(a,b],(a,b),[a,b]$ are said to be of type Lr,lR,lr and LR, respectively.
A singleton is of type LR. For any $x_1, x_2 \in \R$ we write $[x_1,x_2]$ for the closed interval spanned by the two points.

Let $K$ be a bounded subset of $\R$ and $T$ be a transformation on $K$, i.e. $T : K \to K$  a set map. An interval $I$
 is a \emph{$T$ interval} when it is a subinterval  of $K$ such that the restriction $T|I$ is a translation,
 i.e. $Tx - x$ is a constant $c$ for $x$ in
$I$ (constant translation number) so, in particular, $T(I) = I + c$ and $\ell(T(I)) = \ell(I)$. Clearly, $T(I)$ has the same type as $I$.

Notice that if $I$ has end-points $a < b$ and $c > 0$ then $I + c \setminus I$ contains the positive interval $(d,b+c)$ with $d = \max(b, a+c)$.
Similarly, if $c < 0$ then $I + c \setminus I$ contains a positive interval. It follows that
\begin{equation}\label{eq2.1}
I \ \text{a} \ T \ \text{interval with} \ T(I) \subset I \quad \Longrightarrow \quad Tx = x \ \text{for all} \ x \in I.
\end{equation}
For $A \subset \N$ $I$ is a \emph{$T^{A}$ interval} when it is a $T^n$ interval for every  $n \in A$. In particular,
$I$ is a \emph{$T^{\N}$ interval} when it is a $T^n$ interval for every positive $n$.

Clearly, any subinterval of a $T^A$  interval is a $T^A$ interval.

For any $x \in  K$ the singleton $\{ x \}$ is  a $T^{\N}$ interval.

\begin{lem}\label{lem2.1} Let $\{I_p : p \in P \}$ be a family of subintervals of $K$ with $\bigcap_p I_p \not= \emptyset$. The subsets
$\bigcap_p I_p$ and $\bigcup_p I_p$ are subintervals of $K$. If $P$ is finite and all the $I_p$'s are of the same type then
both the intersection and the overlapping union are of the same type as well.

If for $A \subset \N$ all of the $I_p$'s are $T^A$ intervals, then  $\bigcup_p I_p$ is a $T^A$ interval.
\end{lem}

\begin{proof} From the no holes condition it is clear that $\bigcap_p I_p$ is an interval. Fix $z \in \bigcap_p I_p$.
If $x_1, x_2 \in \bigcup_p I_p$ then $[x_1,x_2] \subset [x_1,z] \cup [x_2,z] \subset \bigcup_p I_p$ and so $\bigcup_p I_p$ is an
interval. If all of the $I_p$'s are $T$ intervals then $Tx_1 - x_1 = Tz - z = Tx_2 - x_2$ and so the points of $\bigcup_p I_p$
have a constant translation number.  The result for $T^A$ is obvious from the result for $T$.

The types result is easy to check when the cardinality of $P$ is two.  The general result follows by induction.
\end{proof}

\begin{lem}\label{lem2.2} Let $T, S$ be transformations on $K$ and $I$ be a subinterval of $K$. Any two of the following implies the third:
\begin{itemize}
\item[(i)] $I$ is a $T$ interval.
\item[(ii)] $I$ is an $S \circ T$ interval.
\item[(iii)] $T(I)$ is an $S$ interval.
\end{itemize} \end{lem}

\begin{proof} For $x \in I$, $STx - x = (STx - Tx) + (Tx - x)$.  It is thus clear that any two of the following implies the third:
\begin{itemize}
\item[(i)] $Tx - x$ is constant for all $x \in I$.
\item[(ii)] $STx - x $ is constant for all $x \in I$.
\item[(iii)] $STx - Tx$ is constant for all $x \in I$.
\end{itemize}
\end{proof}

\begin{df}\label{def2.3} For $T$ a transformation on $K$, a bounded subset of $\R$, and $A \subset \N$ the
equivalence relation $E^A_T \subset K \times K$ is defined by
\begin{equation}\label{eq2.2}
E^A_T  =  \{ (x,y)  : \text{there exists a} \ T^{A} \ \text{interval} \ I \ \text{with} \ x,y \in I \}.
\end{equation}
When $A = \N$ we will omit the superscript.
\end{df}

Since $\{x\}$ is a $T^{\N}$ interval, the relation is reflexive.  It is obviously symmetric and transitivity follows from Lemma \ref{lem2.1}.
By Lemma \ref{lem2.1} again, the equivalence class $E^A_T(x)$ is the union of all the $T^{A}$ intervals which contain $x$ and so it is the maximum
$T^A$ interval which contains $x$.

From Lemma \ref{lem2.2} it follows that $T^n(E_T(x))$ is a $T^{\N}$ interval containing $T^nx$. So for all $x \in K, n \in \N$
 \begin{equation}\label{eq2.3}
 T^n(E_T(x)) \ \subset \ E_T(T^nx).
 \end{equation}

%
%
%
%
%
%

 \begin{theo}\label{theo2.5} If for $T$  a transformation on $K$ there exists a positive $T^{\N}$ interval,
 then there exists a positive subinterval $I$ of $K$
 and a positive integer $n$ such that $T^nx = x$ for all $x \in I$. \end{theo}

 \begin{proof} This is a baby version of the Poincar\'{e} Recurrence Theorem.

 The hypothesis says that there exists $x \in X$ such that $E_T(x)$ is a positive interval and so has
 $\ell = \ell( E_T(x)) > 0$. From (\ref{eq2.3}) we have
 \begin{equation}\label{eq2.5}
 \ell \ =  \ \ell(E_T(x))\ = \ \ell(T^n(E_T(x))) \ \leq \ \ell(E_T(T^nx)).
 \end{equation}

 Since $K$ is bounded, the sequence $\{ E_T(T^nx) \}$ cannot be pairwise disjoint. So there exist $n, k \in \N$ such that
 $E_T(T^kx)$ and $E_T(T^{n+k}x)$ intersect and so are equal. By (\ref{eq2.3}) again the positive $T^{\N}$ interval $I = E_T(T^kx)$
 is mapped into itself by $T^n$. It follows from (\ref{eq2.1}) that $T^nx = x$ for all $x \in I$.
\end{proof}

\section{Interval Exchange Transformations }\label{sec3}

We review from \cite{Ke} the definition of an interval exchange transformation.
Let $J = [0,1)$ and $\tilde J = (0,1]$. With $n \geq 2$ let $[n] = \{ 1, \dots, n \}$.  Given  $\aa = (a_1, \dots , a_n)$ a
positive probability vector, i.e.
 $a_i > 0$ for $i \in [n]$ and $\sum_{j=1}^n a_j = 1$, we define
$ b_0 = 0$ and for $i \in [n]$
\begin{equation}\label{eq3.1}
b_i = \sum_{j=1}^i a_j, \qquad J_i = [b_{i-1},b_i), \qquad \tilde J_i = (b_{i-1},b_i].
\end{equation}
  Thus, $J_1, \dots, J_n$ is a partition of $J$ by Lr intervals and $\tilde J_1, \dots, \tilde J_n$ is a partition of $\tilde J$ by lR intervals.
 Let $D = \{ b_1, \dots, b_{n-1} \}$.

 If $\t$ is a permutation of $[n]$, then $\aa^{\t} = (a_{\t^{-1}(1)}, a_{\t^{-1}(2)}, \dots, a_{\t^{-1}(n)})$  is a positive probability vector and
 we form the corresponding $b^{\t}_i, J^{\t}_i$ and $\tilde J^{\t}_i$ for $i \in [n]$. Thus,
 $b^{\t}_{\t(i)} - b^{\t}_{\t(i)-1} = a^{\t}_{\t(i)} = a_i = b_i - b_{i-1}$ for
 all $i \in [n]$ and $D^{\t} =  \{ b^{\t}_1, \dots, b^{\t}_{n-1} \}$.

  We define the transformations $T$ on $J$ and $\tilde T$ on $\tilde J$
 so that for $i \in [n]$, $T : J_i \to J^{\t}_{\t(i)}$ and $\tilde T : \tilde J_i \to \tilde J^{\t}_{\t(i)}$ are the translations given  by
 \begin{equation}\label{eq3.2}
  x \mapsto  x - b_{i-1} + b^{\t}_{\t(i)-1}.
  \end{equation}

  We call $T$ the Lr $(\aa,\t)$ interval exchange transformation and $\tilde T$ the lR $(\aa,\t)$ interval exchange transformation dual to $T$.

  These possess
  the following properties.
  \begin{itemize}
  \item $T$ and $\tilde T$ are invertible and their inverses are the corresponding $(\aa^{\t},\t^{-1})$ transformations.

  \item $T$ is everywhere continuous from the
  right and $\tilde T$ is everywhere continuous from the left.

  \item $T = \tilde T$ on $[0,1] \setminus (D \cup \{0,1 \})$ and is continuous on this open set.
  \end{itemize}

  So we have for $i \in [n]$
  \begin{equation}\label{eq3.3}\begin{split}
  \tilde Tb_i \ = \ \lim_{x \uparrow b_i} Tx \ = \ b^{\t}_{\t(i)}, \hspace{.5cm}\\
  Tb_{i-1}  \ = \ \lim_{\ \ x \downarrow b_{i-1}} \tilde Tx \ = \ b^{\t}_{\t(i)- 1 }.
\end{split}\end{equation}

Notice that if $\{ x_i \}$ is a monotone increasing sequence in $(0,1)$ then it is eventually contained in some $J_i$ and so eventually
$\{ Tx_i \}$ is monotone increasing. It follows that, inductively, for all $k \in \Z$, $T^k$ is everywhere continuous from the
  right  and similarly $\tilde T^k$ is everywhere continuous from the left and both of these preserve eventually monotonicity of sequences.

If $\t(\{1,\dots,j \}) = \{1,\dots,j \}$ for some $j = 1, \dots, n-1$ then $T$ can be decomposed into
an interval exchange transformation on $[0,b_j)$ and one on $[b_j,1)$ (and similarly for $\tilde T$). Keane calls $\t$ \emph{irreducible} when
\begin{equation}\label{eq3.4}
\t(\{1,\dots,j \}) \not= \{1,\dots,j \}\ \text{ for } \ j = 1, \dots, n-1
\end{equation}

Notice that $T0 = 0$ only when $b^{\t}_1 = a^{\t}_{\t^{-1}(1)} = a_1$.  That is, when $\t(1) = 1$.  So if $\t$ is
irreducible, then $T0 \not= 0$ and so $0 \in T(D)$. Similarly, $\tilde T1 = 1$ when $\t(n) = n$ and so $\t(\{ 1, \dots, n-1 \}) = \{ 1, \dots, n-1 \}$.
Thus, irreducibility implies $\tilde T1 \not= 1$ and $1 \in \tilde T(D)$.

If for some $j = 1, \dots, n-1$ we have $b^{\t}_{\t(j)} =  b^{\t}_{\t(j+1)- 1 }$ or, equivalently, $\t(j+1) = \t(j) + 1$, then
$T$ is continuous at $b_j$.  In fact, the interval $J_j \cup J_{j+1}$ is a $T$ interval and $T$ is an interval exchange transformation
with the $n-1$ probability vector $(a_1, \dots, a_j + a_{j+1}, \dots a_n)$. We will call $\t$ \emph{split} when
\begin{equation}\label{eq3.5}
\t(j+1) \not= \t(j) + 1 \ \text{ for } \ j = 1, \dots, n-1.
\end{equation}

If $\t$ is split, then $Tb \not= \tilde Tb $ for $b \in D$ and so $D$ is the set of points at which $T$ and $\tilde T$ are not continuous.

We will call $\t$ \emph{fully split} if, in addition,
\begin{equation}\label{eq3.5a}
\t(1) \not= \t(n) + 1 \ \text{and} \ \t^{-1}(1) \not= \t^{-1}(n) + 1
\end{equation}

These conditions exclude the possibilities $T0 = \tilde T1$ and  $T^{-1}0 = \tilde T^{-1}1$, respectively.

If $\t$ is irreducible, split or fully split then $\t^{-1}$ satisfies the corresponding property.


\begin{lem}\label{lem3.1} If $x \in (0,1)$ and $k \in \N$, then
\begin{equation}\label{eq3.6}
\begin{split}
\{ x, Tx, \dots, T^{k-1}x \} \cap (\{ 0 \} \cup D) = \emptyset \ \Longleftrightarrow \hspace{1cm}\\
 \{ x, \tilde Tx, \dots, \tilde T^{k-1}x \} \cap (\{ 1 \} \cup D) = \emptyset,\hspace{1.5cm}
\end{split}\end{equation}
in which case, $T^jx = \tilde T^jx$ for $j = 1, \dots k$, and
\begin{equation}\label{eq3.6a}
T^jx \in J_i \quad \Longleftrightarrow \quad \tilde T^jx \in \tilde J_i
\end{equation}
for $j = 0,1, \dots, k-1$ and $i \in [n]$.
 \end{lem}

\begin{proof} This is obvious for $k = 1$ and then follows easily by induction.
\end{proof}

\begin{prop}\label{prop3.2} Let $k \in \N$ and let $j_p \in [n]$ for $ 0 \leq p \leq k - 1$. Define
\begin{equation}\label{eq3.7}
I = \bigcap_{p=0}^{k-1} T^{-p}(J_{i_p}), \quad \tilde I = \bigcap_{p=0}^{k-1} \tilde T^{-p}(\tilde J_{i_p}).
\end{equation}
If either $I$ or $\tilde I$ is nonempty, then for some $a < b \in [0,1]$ $I = [a,b),\tilde I = (a,b]$ and for all  $ p = 0, \dots , k$
$T^p$ is a translation on $I$ and $\tilde T^p$ is a translation on $\tilde I$ with the same translation constant. For all  $ p = 0, \dots , k$
on $(a,b)$ $T^p = \tilde T^p$  and these are continuous there. The endpoints satisfy
\begin{equation}\label{eq3.8}
a \in \{ 0 \} \cup \bigcup_{p = 0}^{k-1} T^{-p}D, \quad  b \in \{ 1 \} \cup\bigcup_{p = 0}^{k-1} \tilde T^{-p}D
\end{equation}

As we vary the sequence $\{ j_p \}\in [n]^k$, the nonempty intervals $I$ form an Lr partition of $[0,1)$, which we will call the \emph{$k$-level partition},
with the intervals $\tilde I$ the corresponding lR partition of $(0,1]$, which we will call the \emph{dual $k$-level partition} 
\end{prop}

 \begin{proof}  Again the result is clear for $k = 1$.  Assuming the result for $k$ we prove it for $k+1$.

 By induction hypothesis $T^k$ is a translation on $I = [a,b)$. If $T^k(I)$ meets $J_{j_k} = [b_{j_k -1},b_{j_k})$ then
 $I \cap T^{-k}(J_{j_k}) = [a',b')$ with
 \begin{equation}\label{eq3.9}
   a' = \begin{cases} a \ \text{if} \ T^ka > b_{j_k -1},\\   T^{-k}b_{j_k -1} \ \text{otherwise} \end{cases}
   \quad b' = \begin{cases} b \ \text{if} \ T^kb < b_{j_k},\\   T^{-k}b_{j_k} \ \text{otherwise} \end{cases}
\end{equation}
  and on it $T \circ T^k$
 is a translation. By induction hypothesis, $\tilde T^k$ has the same translation constant on $(a,b]$ as $T^k$ does on $[a,b)$. So
 $\tilde I \cap \tilde T^{-k}(\tilde J_{j_k}) = (a',b']$ and on it $\tilde T \circ \tilde T^k$ has the same translation constant as
 $T \circ T^k$ does on $[a',b')$. Finally, $T \circ T^k = \tilde T \circ \tilde T^k$ on $(a',b')$.

The endpoint result (\ref{eq3.8}) follows inductively from (\ref{eq3.9}).

Distinct sequences yield disjoint intervals $I$ and every point of $(0,1]$ lies is some $I$. Hence, they form an Lr partition.
 \end{proof}

 \begin{cor}\label{cor3.2a} If the Lr $(\aa,\t)$ interval exchange transformation $T$ or its dual $\tilde T$ has a periodic point of period $k$,
 then there is a positive subinterval of $(0,1)$ on which $T^k = \tilde T^k = id$. \end{cor}

\begin{proof} Assume that $T^kx = x$ for some $x \in [0,1)$. If $I = [a,b)$ is the element of the $k$ level partition which contains $x$ then since $T^k$ is
a translation on $I$ and so $T^ky - y = T^kx - x = 0$ for $y \in [a,b)$. Since $\tilde T^k = T^k$ on $(a,b)$ we see that
$\tilde T^k = T^k = id$ on $(a,b)$. Similarly, $\tilde T^kx = x$ for $x \in \tilde I = (a,b]$ implies $\tilde T^k = T^k = id$ on $(a,b)$.
\end{proof}

\begin{df} We say that the interval exchange transformation $T$ satisfies the \emph{Keane Condition} when for all $k \in \N$ $D \cap T^k(D) = \emptyset$.
\end{df}

Of course, $D \cap T^k(D) = \emptyset$ if and only if $T^{-k}(D) \cap D = \emptyset$.
The Keane Condition says exactly that the $T$ orbits of the points of $D$
are infinite and distinct.

\begin{lem}\label{lem3.3} $D \cap T^k(D) = \emptyset$ for all $k \in \N$ if and only if
$D \cap \tilde T^k(D) = \emptyset$ for all $k \in \N$. \end{lem}

\begin{proof} Assume that $a_1 \in D$ and $T^ka_1 = b_{i-1} \in D$ so that $1 < i < n$. Let $I$ be an interval  of the $k$-level Lr interval described in Proposition \ref{prop3.2}
 with
$a_1 \in I$. From equation (\ref{eq3.9}) $I$ has $a_1$ as left end
point. It is translated by $T^k$ to an interval with $b_{i-1}$ as left endpoint. So  some  interval $\tilde K = (c,d] $ of the dual $k$-level
lR partition
is translated by $\tilde T^k$ so that $\tilde T^kd = b_{i-1}$. By (\ref{eq3.8}) $d \in \{ 1 \} \cup\bigcup_{p = 0}^{k-1} \tilde T^{-p}D$.
That is, $\tilde T^{k-p}a_2 = b_{i-1}$ for some $a_2 \in D$ and some $p < k$ or else $\tilde T^k1 = b_{i-1}$. In the latter case, $1$ is not fixed and
so $1 \in \tilde T(D)$ and $b_{i-1} \in \tilde T^{k+1}(D)$. Thus, $D$ meets $\tilde T^j(D)$ for some $j$ with $1 \leq j \leq k+1$.

Similarly, $D \cap \tilde T^k(D) \not= \emptyset $ implies $D \cap T^j(D) \not= \emptyset$ for some $j \in \N$.
\end{proof}

The major application of this condition is Keane's Theorem:

\begin{theo}\label{theo3.4} If the Lr $(\aa,\t)$ interval exchange transformation $T$ is  irreducible and satisfies the Keane Condition, then
the maps $T$ and $\tilde T$ are minimal. That is, for every $x \in [0,1)$ the orbit $\{ T^kx : k \in \Z \}$ is dense in $[0,1)$ and
for every $x \in (0,1]$ the orbit $\{ \tilde T^kx : k \in \Z \}$ is dense in $(0,1]$. \end{theo}

\begin{proof} In \cite{Ke} Keane proves this for $T$. The results for $\tilde T$ follow because the $\tilde T$ lR interval exchange transformation
is conjugate to an Lr interval exchange transformation via the map $x \mapsto 1 - x$. From Lemma \ref{lem3.3} it follows that the
conjugate also satisfies the Keane Condition.

Irreducibility says that no interval $[0,b_i)$ is invariant for $T$ and no $(0,b_i]$ is invariant for $\tilde T$.  Thus, the
conjugate Lr transformation is irreducible and Keane's Theorem applies to it. Minimality is preserved by conjugation.
\end{proof}

For constructing examples the following Irrationality Theorem from \cite{Ke} is useful.

\begin{theo}\label{theo3.5} If $\t$ is irreducible and $\{ a_1, \dots, a_{n-1}, a_n \}$ is linearly independent over the field of rationals $\Q$
(or, equivalently,
 $\{ a_1, \dots, a_{n-1}, 1 \}$ is linearly independent over $\Q$) then the $(\aa,\t)$ interval exchange transformation satisfies the
 Keane Condition. \end{theo}

\section{The Associated Subshift and Its Asymptotic Pairs}\label{sec4}

Let $T, \tilde T$ be the interval exchange transformations associated with $(\aa, \t)$ as described in the previous section.

We define the \emph{itinerary functions} $\I : [0,1) \to \Omega, \tilde \I : (0,1] \to \Omega$ by
\begin{equation}\label{eq4.2a}
\begin{split}
 \I(x)_k = i \ \Longleftrightarrow \ T^kx \in J_i \ \ \text{for} \ k \in \Z, i \in [n], x \in [0,1), \\
  \tilde \I(x)_k = i \ \Longleftrightarrow \ \tilde T^kx \in \tilde J_i \ \ \text{for} \ k \in \Z, i \in [n], x \in (0,1].
 \end{split}
 \end{equation}

Clearly we have the following commutative diagrams showing that $\I$ maps $T$ and $\tilde \I$ maps $\tilde T$ to the shift $S$ on $\Omega$.

\[
\begin{array}{cc}
\xymatrix{
\Omega  \ar[r]^{S} & \Omega  \\
[0,1) \ar[u]^{\I} \ar[r]_T & [0,1) \ar[u]_{\I}
}
\,\,\,\,\,\,&\,\,\,\,\,\,
\xymatrix{
\Omega  \ar[r]^{S} & \Omega  \\
(0,1] \ar[u]^{\tilde \I} \ar[r]_{\tilde T} & (0,1] \ar[u]_{\tilde \I}
}
\end{array}\]

Define
\begin{equation}\label{eq4.1}
D_{\infty} \ = \ \{0, 1 \} \cup \bigcup_{k \in \Z} T^k(D)\ = \ \{0, 1 \} \cup \bigcup_{k \in \Z} \tilde T^k(D).
\end{equation}

Notice that either $0 \in T(D)$ or else $T0 = 0$. Similarly, $1 \in \tilde T(D)$ or else $\tilde T1 = 1$. Also,
$D^{\t} \subset D_{\infty}$.

It follows that $R = [0,1] \setminus D_{\infty}$ is a dense $G_{\delta}$ subset of $[0,1]$ on which
 $T = \tilde T$ and the set is invariant, i.e. $T(R) = \tilde T(R) = R$. Hence, $T^{-1} = \tilde T^{-1}$ on $R$ as well.
   The points of $R$ are called the \emph{regular points} for the transformation.
Since $R$ is disjoint from $D \cup D^{\t}$, it follows that $T$ and its inverse are continuous at the points of  $R$.

\begin{prop}\label{prop4.1} The map $\I$ is everywhere continuous from the
  right and $\tilde \I$ is everywhere continuous from the left. On $R$ the maps $\I$ and $\tilde \I$ are equal and are continuous. \end{prop}

  \begin{proof} On each interval $J_i = [b_{i-1},b_i)$ the function $\I(\cdot)_0$ is constantly $i$ as is $\tilde \I(\cdot)_0$ on $\tilde J_i$.
  So $\I(\cdot)_0$ is continuous from the right and $\tilde \I(\cdot)_0$  is continuous from the left. They agree and are continuous
  on the open interval $(b_{i-1},b_i)$.

  For every $k \in \Z$ $\I(T^kx)_0 = (S^k\I(x))_0 = \I(x)_k$.  That is, $\I(\cdot)_k = \I(\cdot)_0 \circ T^k$. Since $T^k$ and $\I(\cdot)_0$
  are continuous from the right and $T^k$ preserves eventually monotonicity,
   it follows from the definition of the product topology on $\Omega$ that $\I$ is continuous from the right.
  Similarly, $\tilde \I$ is continuous from the left. Furthermore, since $R$ is invariant, it follows that $\I = \tilde \I$ on $R$ and so at points
  of $R$ the map is continuous from both directions.
  \end{proof}

  Define
\begin{equation}\label{eq4.2}
X_0  =  \I(R) = \tilde \I(R) \quad \text{and} \quad X = \overline{X_0}.
\end{equation}

The set $X_0$ is an $S$ invariant subset of $\Omega$, i.e. $S(X_0) = X_0$, and so the closure $X$ is a closed, invariant subset. That is,
$(X,S)$ is a subshift of the full shift $(\Omega,S)$ on $n$ symbols.

Keane provides an explicit description of $X$ as follows. Let $\J = ([0,1) \times \{ 0 \}) \cup ((D_{\infty} \setminus \{ 0 \}) \times \{-1 \}) $.
Order $\J$ lexicographically and use the order topology. Since every nonempty subset has a supremum and infimum, it follows that
$\J$ is a compact Hausdorff space. Identify each $x \in [0,1)$ with $(x,0)$ in $\J$ and for $x \in D_{\infty} \setminus \{ 0 \}$ write
$x-$ for $(x,-1)$. Notice that each $x-,x$ with $x \in D_{\infty} \setminus \{ 0,1 \}$ is a gap pair, i.e. $x- < x$ and there are no points
between them. If $D_{\infty}$ is dense, as in the minimal case, then $\J$ is homeomorphic to the Cantor Set. The projection to the first
coordinate defines a continuous, order-preserving surjection from $\J$ onto $[0,1]$  which is one-to-one over the points of $R \cup \{ 0,1 \}$
and which maps $x$ and $x-$ to $x$.

Let $\J_i = [b_{i-1},b_i-]$ for $i \in [n]$. This is a partition of $\J$ by $n$ clopen intervals. Define the homeomorphism $\T$ on $\J$
by mapping $\J_i$ to the clopen interval $\J^{\t}_{\t(i)}$ in the order-preserving way which extends $T$ on $J_i$.
Now the itinerary map is a continuous map from $\J$ onto $X$, mapping $\T$ to $S$ and extending $\I$ on the dense set $R$.

Our alternative description of $X$ is the following

\begin{theo}\label{theo4.2} $X$ equals the union $ \I([0,1)) \cup \tilde \I((0,1])$.\end{theo}

\begin{proof} Any $x \in [0,1)$ is a limit of a decreasing sequence $\{ x^i \}$ in $R$. So $\I(x) = \lim \I(x^i)$ is in the closure of $X_0$.
Similarly, $\tilde \I((0,1]) \subset X$.

Now suppose $\{ \a^i = \I(x^i) = \tilde \I(x^i) \}$ is a sequence in $X_0$ which converges to $\b \in \Omega$. By going to a subsequence we can assume that
$\{ x^i \}$ converges monotonically to a point $y \in [0,1]$. If the sequence is decreasing, then $\b = \I(y)$.  If the sequence is increasing, then
$\b = \tilde \I(y)$.
\end{proof}

\begin{cor}\label{cor4.3} If the Lr $(\aa,\t)$ interval exchange transformation $T$ is  irreducible and satisfies the Keane Condition, then
$(X,S)$ is a minimal subshift.\end{cor}

\begin{proof} By Keane's Theorem \ref{theo3.4} the closure of every $T$ orbit and every $\tilde T$ orbit contains $R$. By Theorem \ref{theo4.2}
and continuity of $\I = \tilde \I$ at points of $R$, it follows that the closure of every $S$ orbit in $X$ contains $X_0$ and so equals $X$.
\end{proof}

\begin{theo}\label{theo4.4} Assume that the Lr $(\aa,\t)$ interval exchange transformation $T$ has no periodic points.
\begin{itemize}
\item[(i)] If $x, y$ are distinct points of $[0,1)$, then
the pair $\I(x), \I(y)$  is neither positively, nor negatively asymptotic.
\item[(ii)] If $x, y$ are distinct points of $(0,1]$ then
the pair $\tilde \I(x), \tilde \I(y)$  is neither positively, nor negatively asymptotic.
\item[(iii)] If $\a, \b$ are distinct points of $X$ and $\a \in X_0$, then the pair is neither
positively nor negatively asymptotic.
\item[(iv)] The maps $\I$ and $\tilde \I$ are injective.
\end{itemize}
\end{theo}

\begin{proof} (i) Assume that $i_k = \I(x)_k = \I(y)_k$ for all $k \geq 0$. By Proposition \ref{prop3.2}
$\{ I_k =  \bigcap_{p=0}^{k-1} T^{-p}(J_{i_p})  \}$ is a decreasing sequence of intervals each of which contains $x$ and $y$.
For $k > p$ it is a $T^p$ interval.
By Lemma \ref{lem2.1} the intersection $I = \bigcap_{k > 0} I_k = \bigcap_{k > p} I_k $ is a $T^p$ interval for all $p$. Since it contains
$x$ and $y$, it is a positive $T^{\N}$ interval if $x \not= y$. Theorem \ref{theo2.5} would then imply the there is an interval on which
$T$ is periodic.  Since $T$ is minimal, it does not admit periodic points and so $x = y$.

Now if $\I(x)_k = \I(y)_k$ for $k \geq N$, then $\I(T^Nx)_k = \I(T^Ny)_k$ for $k \geq 0$ and so $T^Nx = T^Ny$.  Because $T$ is injective.
it follows that $x = y$.

Apply the result to $T^{-1}$ to see that the pair is negatively asymptotic only when $x = y$.

(ii) As in (i) $\tilde \I(x)_k = \tilde \I(y)_k$ for $k \geq N$ with $x \not= y$ implies that there is a positive interval on which $\tilde T$ is periodic.
Such an interval meets $R$ and the intersection would consist of periodic points for $T$. The contradiction shows that $x = y$.

(iii) Assume that $\a = \I(x) = \tilde \I(x)$ for some $x \in R$. If $\b = \I(y)$, then $x \not= y$ and so $\a = \I(x)$ and $\b = \I(y)$
are not asymptotic. Similarly, if $\b = \tilde \I(y)$, then $x \not= y$ and so $\a = \tilde \I(x)$ and $\b = \tilde \I(y)$ are not asymptotic.
By Theorem \ref{theo4.2} one of these two cases applies.

(iv) If for $x \not= y$ in $[0,1)$, $\I(x) = \I(y)$ then the ``pair'' $\I(x),\I(y)$ is asymptotic contradicting (i).  Similarly, for $\tilde \I$.
\end{proof}

The result (i) above is equivalent to a comment at the end of Section 5 of \cite{Ke}.

The main result of this paper is the description of asymptotic pairs.

\begin{theo}\label{theo4.5} Assume that the  Lr $(\aa,\t)$ interval exchange transformation $T$ is irreducible and satisfies the
Keane Condition. Let $\a, \b \in X$.
\begin{itemize}
\item[(i)] If $\a \not= \b$ then the pair  $\{ \a, \b  \}$ is positively asymptotic if and only if there exist $i \in [n]$ and $k \in \Z$ such that
with $x =  b^{\t}_i \in D^{\t}$, $\{ \a, \b \} = \{ S^k \I(x), S^k \tilde \I(x) \}$.
\item[(ii)] If $\a \not= \b$ then the pair $\{ \a, \b  \}$ is negatively asymptotic if and only if there exist $i \in [n]$ and $k \in \Z$ such that
with $x =  b_i \in D$, $\{ \a, \b \} = \{ S^k \I(x), S^k \tilde \I(x) \}$.
\item[(iii)] If $\I(x) = \tilde \I(y)$ for some $x \in [0,1), y \in (0,1]$, then $x = y$ and the common point lies in $R$.
\item[(iv)] Assume, in addition, that $\t$ is fully split. If the pair $\{ \a, \b  \}$ is doubly asymptotic, then $\a = \b$. Conversely, if $\t$ is not
split, then there exist doubly asymptotic pairs of distinct points.
\end{itemize}
\end{theo}

\begin{proof} By the Keane Condition every $T$ orbit or $\tilde T$ of a point of $D$ meets $D$ just once. Because $\t$ is irreducible
$0 \in T(D)$ and $1 \in \tilde T(D)$.

(i) First, assume that $x \in D^{\t}$. Let $b_{i-1} = T^{-1}x$, the left endpoint of some $J_i$ and, $b_j = \tilde T^{-1}x$
the right end-point of some $\tilde J_j$. Furthermore, $i \not= j$ since the endpoints $b_{i-1}$ and $b_i$ are translated by $T$ and
$\tilde T$ to the endpoints of the interval $J^{\t}_{\t(i)}$. If
$\t$ is split, then $b_{i-1} \not= b_j$. But in any case, $ \I(x)_{-1} = i \not= j = \tilde \I(x)_{-1}$. Thus, letting
$\a = \I(x), \b = \tilde \I(x)$. We have $\a \not= \b$.

By the Keane Condition and irreducibility, $T^kx \not\in \{ 0 \} \cup D$ for $k \geq 0$. So
Lemma \ref{lem3.1} implies that $T^kx = \tilde T^kx$ and
$ \I(x)_{k}  = \tilde \I(x)_{k}$ for all $k \geq 0$. Thus, the pair $\a, \b$ is positively asymptotic.

Conversely, assume that
$\a = \I(x)$ and $\b = \tilde \I(y)$ are positively asymptotic. By shifting down we may assume that $\I(x)_k =  \tilde \I(y)_k$ for all $k \geq 0$ and
that $T^kx \not\in \{0 \} \cup D$ for all $k \geq 0$. Again Lemma \ref{lem3.1} implies that $T^kx = \tilde T^kx$ for all $k \geq 0$.
It follows that $\tilde \I(x)_k = \I(x)_k = \tilde \I(y)_k$ for all $k \geq 0$. From Theorem \ref{theo4.4} (ii) it follows that $x = y$.
So $\a = \I(x)$ and $\b = \tilde \I(x)$. If $\a \not= \b$, $x \not\in R$ and so the $T$ orbit of $x$ hits $\{ 0 \} \cup D$.

Now we can  move in the negative direction and shift so that we may assume that $T^{-1}x \in \{ 0 \} \cup D$ and $T^kx \not\in \{ 0 \} \cup D$
for all $k \geq 0$. So again $T^kx = \tilde T^kx$ for all $k \geq 0$. Note that if $T^{-1}x = 0$ then $T^{-2}x \in D$. Thus,
$x \in T( \{ 0 \} \cup D) \setminus \{ 0 \} = D^{\t}$.

(ii) Apply (i) to $T^{-1}$ which is also irreducible and satisfies the Keane Condition.

(iii) If  $\I(x)  = \tilde \I(y)$, then the pair is trivially positively asymptotic. The proof of (i) essentially shows that $T^Nx = T^Ny$ for
some $N \in \N$ and so $x = y$. If the orbit of $x$ hits $D$ then the proof of (i) shows that $\I(x) \not= \tilde \I(x)$. It follows that
a point $x$ with $\I(x)  = \tilde \I(x)$ must lie in $R$.

(iv) We use the following
\begin{itemize}
\item[(a)] If $d \in D$, then $T^{-k}d \not\in \{ 0 \} \cup D^{\t}$ for $k \geq 0$ and $T^{-k}d \not\in \{ 0 \} \cup D$ for $k \geq 1$
by irreducibility and the Keane condition.
By Lemma \ref{lem3.1} applied to the inverse transformations, $T^{-k}d = \tilde T^{-k}d$ for $k \geq 0$ and $\I(d)_{-k} = \tilde \I(d)_{-k}$
for $k \geq 1$. \vspace{.5cm}

\item[(b)]By irreducibility $T^{-1}0 \in D$ so by (a)  $T^{-k}(T^{-1}0)= \tilde T^{-k}(T^{-1}0)$
and for $k \geq 1$ $\I(T^{-1}0)_{-k} = \tilde \I(T^{-1}0)_{-k}$. Similarly, for
$k \geq 0$, $T^{-k}(\tilde T^{-1}1)= \tilde T^{-k}(\tilde T^{-1}1)$ and for $k \geq 1$ $\I(\tilde T^{-1}1)_{-k} = \tilde \I(\tilde T^{-1}1)_{-k}$.

\end{itemize}

Now suppose that $\t$ is not split. Thus, for suitable $x \in D^{\t}$ we have $b_{i-1} = T^{-1}x$ and $b_j = \tilde T^{-1}x$ with
$j = i-1$ so that $d = T^{-1}x = \tilde T^{-1}x \in D$. Thus, $T^{k}x = \tilde T^{k}x $ for all $k \in \Z$. We have
$i = \I(x)_{-1}, j = i-1 =  \tilde \I(x)_{-1}$. On the other hand, the proof of (i) and (a) above imply that
$ \I(x)_{k}  =  \tilde \I(x)_{k}$ for all $k \not= -1$ in $\Z$. That is,
$\I(x)$ and $\tilde \I(x)$ is a pair of distinct doubly asymptotic points.

Also, if $T0 = \tilde T1 = x$ and $T^{-1}0 = \tilde T^{-1}1$, then we can similarly see that $ \I(x)_{k}  =  \tilde \I(x)_{k}$ for all $k \not= -1, -2$ in $\Z$.
$\I(x)$ and $\tilde \I(x)$ is a pair of distinct doubly asymptotic points.

Now assume that $\t$ is fully split. We begin with a positively asymptotic pair and after shifting we can assume the
pair is $\I(x), \tilde \I(x)$ with $x \in D^{\t}$.

Since $\t$ is  split,  $d_1 = b_{i-1} = T^{-1}x$ and $d_2 = b_j = \tilde T^{-1}x$ are not equal and
$d_1 \in \{ 0 \} \cup D, d_2 \in \{ 1 \} \cup D$. Since $\t$ is fully split,  we do not have both $d_1 = 0$ and $d_2 = 1$.

Let $y_1 = T^{-1}d_1, y_2 = \tilde T^{-1}d_2$. From (a) and (b), we have
$T^{-k}y_1 = \tilde T^{-k}y_1$ and $T^{-k}y_2 = \tilde T^{-k}y_2$ for $k \geq 0$
 and for $k \geq 1$
\begin{equation}\label{eq4.3}
\begin{split}
\tilde \I(y_1)_{-k} = \I(y_1)_{-k} = \I(x)_{-k-2}, \\
\I(y_2)_{-k} = \tilde \I(y_2)_{-k} = \tilde \I(x)_{-k-2}.
\end{split}
\end{equation}

If $\I(x)$ and $\tilde \I(x)$ were negatively asymptotic, then from Equation (\ref{eq4.3}) it would follow that $\I(y_1)$ and $ \I(y_2)$
are negatively asymptotic.  It would then follow from Theorem \ref{theo4.4} that $y_1 = y_2$.

Now either $d_1 \in D$ or $d_2 \in D$. Suppose $d_1 \in D$.  then $ \tilde T^{-1}d_1 = T^{-1}d_1 = y_1 = y_2 = \tilde T^{-1}d_2$.
Applying $\tilde T$ we obtain $d_1 = d_2$ which is not true. Similarly if $d_2 \in D$.

Thus, $\I(x)$ and $\tilde \I(x)$ are not negatively asymptotic.
Consequently, pairs of distinct points of $X$ are never doubly asymptotic.
\end{proof}

Thus, we have obtained our promised result.

\begin{theo}\label{theo4.6} If an interval exchange transformation is  irreducible, fully split and satisfies the Keane Condition, then no pair of
distinct points is doubly asymptotic for the associated subshift. \end{theo}

We conclude by describing the simplest examples.

$n = 2$: The transposition $\t = (1,2)$ is irreducible and split.  If $a_1$ is chosen irrational,
then by Theorem \ref{theo3.5} the $(\aa,\t)$ interval exchange map on two intervals satisfies the Keane Condition. If we identify $0$ with $1$ then the maps
$T$ and $\tilde T$ become irrational rotations of the circle, or, equivalently, the translation by $a_1$ on the quotient $\R/\Z$. The lift to
$(X,S)$ yields a Sturmian subshift. The permutation
$\t$ is not fully split. In fact,
$T0 = \tilde T1$ and $T^{-1}0 = \tilde T^{-1}1$. In this case, $b_1 = a_1$ and, as is demonstrated in (iv) above,
 the orbit of the pair $\I(b_1), \tilde \I(b_1)$ is the
unique pair of orbits which are positively or negatively asymptotic and it is doubly asymptotic.

$n = 3$: We choose $a_1, a_2$ so that $\{ a_1, a_2, 1 \}$ is linearly independent over $\Q$.
The irreducible permutations are the three cycle $(1,2,3)$ (and its inverse) and the transposition $(1,3)$. In each case the $(\aa,\t)$ transformation
satisfies the Keane Condition by Theorem \ref{theo3.5} again.

The three cycle $(1,2,3)$ is not split. Both pairs $\I(b_1), \tilde \I(b_1)$ and $\I(b_2), \tilde \I(b_2)$ are doubly asymptotic. The maps $T$ and
$\tilde T$ are the same as those for the $n = 2$ case above with probability vector $(a_1 + a_2,a_3)$. The subshift is obtained from
the translation by $a_1 + a_2$ on  $\R/\Z$, by cutting at the points of two orbits instead of just one. See Keane's point-doubling construction
of $X$ described above.

On the other hand, the transposition  $\t = (1,3)$  is irreducible and fully split.
For such an $(\aa,\t)$ transformation no pair of
distinct points is doubly asymptotic for the associated subshift $(X,S)$.

\bibliographystyle{amsplain}

\end{document}